\documentclass[10pt,  b5paper, reqno]{article}
\usepackage{amsmath, amsthm, amsfonts, amssymb}
\usepackage[dvips=false,pdftex=false,vtex=false]{geometry}
\def\R{{\mathbb R}}

\def\E{{\mathbb E}}

\def\U{{\mathcal U}}
\def\V{{\mathcal V}}
\def\W{{\mathcal W}}

\def\A{{\mathcal A}}
\def\l{{\ell}}

\def\dto{{\overset{d}{\longrightarrow }}}
\DeclareMathOperator{\supp}{supp}
\DeclareMathOperator{\wt}{wt}

\long\def\symbolfootnote[#1]#2{\begingroup%
\def\thefootnote{\fnsymbol{footnote}}\footnote[#1]{#2}\endgroup} 

\newtheorem{thm}{Theorem}[section]

\newtheorem{lem}[thm]{Lemma}

\theoremstyle{definition}
\newtheorem{defn}[thm]{Definition}

\theoremstyle{remark}

\title{Central limit theorem for moments of spectral measures of Wigner matrices}
\author{Trinh Khanh Duy}
\begin{document}
\maketitle
\begin{abstract}
Spectral measures of Wigner matrices are investigated. The Wigner semicircle law for spectral measures is proved. Regard this as the law of large number, the central limit theorem for moments spectral measure is also derived. The proof is based on moment method and combinatorial method.  
\end{abstract}

\symbolfootnote[0]{{2000 Mathematics Subject Classification }. Primary 60F05; secondary 15A52}

%Keywords. Wigner matrix, semicircle law, spectral measure, central limit theorem,  
\section{Introduction}
This paper concerns with real Wigner matrices $X_N$ of the form
\[
	X_N(j,i) = X_N(i,j) := \frac{\xi_{ij}}{\sqrt{N}},\quad 1\le i \le j \le N.
\]
Here $\{\xi_{ii}\}_{1\le i}$ and $\{\xi_{ij}\}_{1\le i < j}$ are two i.i.d.\ (independent identically distributed) sequences of mean zero (real) random variables. We require in addition that all moments of $\xi_{11}$ and $\xi_{12}$ are finite and $\xi_{12}$ has unit variance, that is, $\E[|\xi_{12}|^2] = 1$.

Let $\lambda_1^{(N)} \le \lambda_2^{(N)} \le \cdots \le \lambda_N^{(N)}$ be the eigenvalues of $X_N$ and  
\[
	L_N := \frac 1N \sum_{i = 1}^N \delta_{\lambda_i^{(N)}}
\]
be the empirical distribution (measure) of $X_N$, where $\delta$ denotes the Dirac measure. Then the Wigner semicircle law claims that as $N$ tends to infinity, $L_N$ converges weakly, in probability, to the semicircle distribution. This means that for any bounded continuous function $f\colon \R \to \R$, $\langle L_N, f \rangle $ converges in probability to $\langle \sigma, f \rangle$. Here the semicircle distribution, denoted by $\sigma$, is the probability distribution supported on $[-2,2]$ with density
\[
	\sigma(x) = \frac{1}{2\pi} \sqrt{4 - x^2}, (-2 \le x \le 2).
\] 

There are many proofs of the Wigner semicircle law. Let us mention here Wigner's original one which based on combinatoric arguments. Since the semicircle distribution $\sigma$ has compact support, in order to prove the Wigner semicircle law, it is sufficient to show that all moments of $L_N$ converges in probability to the corresponding moments of $\sigma$, namely, for $k = 0,1,2,\dots,$
\begin{equation}\label{LLN-empirical}
	\langle L_N, x^k\rangle \to \langle \sigma, x^k\rangle \text{ in probability as $N \to \infty$. }
\end{equation} 
The $k$th moment of $L_N$ can be written as 
\[
	\langle L_N, x^k \rangle = \frac{1}{N} \sum_{j = 1}^N (\lambda_j^{(N)})^k = \frac{1}{N} \sum_{j = 1}^N X_N^k(j,j).
\]
Thus, in some respects, the semicircle law states that the average of the diagonal elements of $X_N^k$ converges in probability to $\langle \sigma, x^k \rangle$.

With a little modification, one can show that each diagonal element of $X_N^k$ does converge to $\langle \sigma, x^k \rangle$ as $N$ tends to infinity. In particular, for $k = 0,1,2,\dots,$
\begin{equation}\label{LLN-spectral}
	X_N^k(1,1) \to \langle \sigma, x^k \rangle \text{ in probability as $N \to \infty$}.
\end{equation}
On the other hand, there is a probability measure $\nu_N$ on $\R$ satisfying 
\[
	\langle \nu_N, x^k \rangle = X_N^{k}(1,1),\quad  k = 0,1,2,\dots,
\] 
called the spectral measure of $(X_N, e_1)$, where $e_1 = (1,0,\dots,0)^T \in \R^N$. It then follows that the spectral measure $\nu_N$ also converges weakly, in probability, to the semicircle distribution because of the compact support of the semicircle distribution.

Regard the convergence in probability of moments as the law of large numbers, the central limit theorem for moments of the empirical distributions $L_N$ has been derived. It is known that scaled by $N$,
\[
	N\left (\langle L_N, x^k \rangle - \E[\langle L_N, x^k \rangle ]\right) 
\]  
converges weakly to the Gaussian distribution whose variance depends on the second and fourth moments of $\xi_{11}$ and $\xi_{12}$. This and the multidimensional version were studied in \cite{Anderson2006}. The main purpose of this paper is to investigate the central limit theorem for moments of the spectral measures $\nu_N$, or just the central limit theorem for diagonal elements $X_N^k(1,1)$. The main result is as follows.

\begin{thm} Let 
\[
	\bar S_{N,k} = \sqrt{N} \left( \langle \nu_N, x^k\rangle - \E[\langle \nu_N, x^k\rangle]\right) = \sqrt{N}\left( X_{N}^k(1,1) - \E[X_{N}^k(1,1)]\right).
\]
Then there exists a sequence of jointly Gaussian random variables $\{\eta_k\}_{k = 2,3,\dots}$ independent of  $\zeta$ which has the same distribution as $\xi_{11}$ such that the following hold.
\begin{itemize}
	\item[\rm (i)] For even $k$, 
		\[
			\bar S_{N,k}  \dto \eta_k \text{ as } N\to \infty.
		\]
	\item[\rm (ii)] For odd $k \ge 3$,
		\[
			\bar S_{N,k}  \dto a_k \zeta + \eta_k \text{ as } N\to \infty,
		\]
		where $a_k$ is a constant.
	\item[\rm (iii)] For fixed $K$, the joint distribution of $(\bar S_{N, 1}, \bar S_{N, 2}, \dots,\bar S_{N,K})$ converges to that of $(\zeta, \eta_2, a_3 \zeta + \eta_3, \dots)$. 
	
\end{itemize}
Here the symbol ``$\dto$'' is used to denote the weak convergence of random variables.
\end{thm}

The moment method is used to prove the central limit theorem. However, to compare with combinatoric arguments in \cite{Anderson2006}, the big difference is that every word starts at $1$, as we will see in the next section. To overcome this difficulty, we refine method in \cite{Anderson2006} using some idea from \cite{Sinai1998}. The paper is organized as follows. Section~2 deals with some combinatorics objects such as Wigner words, CLT sentences and key combinatoric arguments. We prove in Section~3 the Wigner semicircle law for spectral measures and investigate the central limit theorem in Section~4.

\section{Words, sentences}
This section deals with basic notions and key combinatoric arguments needed in the paper.

	We begin with the definition of words. A word $w=\{s_1, s_2,\dots, s_k\}$ is a finite sequence of positive integer numbers called letters. A word is closed if the first and the last letters are the same. The length of $w$ is denoted by $\l (w) := k$. The support, denoted by $\supp(w)$, is the set of letters appearing in $w$, and the weight, $\wt(w)$, is defined as the cardinality of $\supp(w)$. If we restrict the condition that $s_1, s_2,\dots, s_k \in \{1,2,\dots, N\}$, we call $w$ an $N$-word, where $N$ is a positive integer number. 
		
	Two words $w_1$ and $w_2$ are called equivalent, denoted by $w_1 \sim w_2$, if there is a bijection from $\supp(w_1)$ onto $\supp(w_2)$, which maps $w_1$ to $w_2$.
	
	A word $w$ is associated with an undirected graph $G_w = (V_w, E_w)$, with $\wt(w)$ vertices $V_w = \supp(w)$ and $(k-1)$ edges $E_w = \{(s_i, s_{i+1}), i = 1,2,\dots, k-1\}$. Then the word $w$ defines a path/walk on the connected graph $G_w$. We define the set of self edges as $E_w^s = \{e \in E_w : e = (u,u), u \in V_w\}$ and the set of connecting edges as $E_w^c = E_w \setminus E_w^s$. For $e\in E_w$, we use $N_e^w$ to denote the number of times this path traverses the edge $e$ (in any direction). Note that equivalent words generate the same graphs (up to graph isomorphism) $G_w$ and the same passage counts $N_e^w$.

A sentence $a = (w_1,w_2, \dots, w_n)$ is a finite sequence of words of at least one word long. The support of $a$ is defined as $\supp(a) = \cup_{i = 1}^n \supp (w_i)$, and the weight of $a$, $\wt(a)$, is just the cardinality of $\supp(a)$. Two sentences $a_1$ and $a_2$ are called equivalent, denoted by $a_1 \sim a_2$, if there is a bijection from $\supp(a_1)$ onto $\supp(a_2)$, which maps $a_1$ to $a_2$. 

A graph $G_a = (V_a, E_a)$ associated with a sentence $a=(w_1,w_2, \dots, w_n)$, where  $w_i = (s_1^i, s_2^i, \dots, s_{\l(w_i))}^i), i = 1,2,\dots,n$, is the graph with vertices $V_a = \supp(a)$ and undirected edges
\[
	E_a = \{(s_j^i, s_{j+1}^i) : j = 1,\dots, \l(w_i) -1, i = 1,2, \dots, n	\}.
\]	
We define the set of self edges as $E_a^s = \{e \in E_a : e=\{u,u\}, u \in V_a\}$ and the set of connecting edges as $E_a^c = E_a \setminus E_a^s$.

In words, the graph associated with a sentence is obtained by piecing together the graphs of the individual words. Thus, the graph of a sentence may be disconnected. Note that the sentence $a$ defines $n$ paths in the graph $G_a$. For $e \in E_a$, we use $N_e^a$ to denote the number of times the union of these paths traverses the edge $e$ (in any direction). We note that equivalent sentences generate the same graphs $G_a$ and the same passage counts $N_e^a$.

The paper deals with closed words starting at $1$. Let $\W^{(N)}$ be the set of all $N$-words starting at $1$. Let $\U^{(N)}:= \{w \in \W^{(N)}: E_w^s = \emptyset\}$ be the subset of $\W^{(N)}$ consisting of words with no self-edge, and $\V^{(N)} := \W^{(N)} \setminus \U^{(N)}$. Set
\[
	\W := \bigcup_{N = 1}^\infty \W^{(N)},\quad  \U := \bigcup_{N = 1}^\infty \U^{(N)},\quad \V := \bigcup_{N = 1}^\infty \V^{(N)}.
\]
Henceforth, the sets $\W_k, \W_k^{(N)},\U_k, \U_k^{(N)},\V_k, \V_k^{(N)}$ with a subscript $k$, are used to denote the corresponding subsets consisting of words of length $k+1$.

A closed word $w$ is called a weak Wigner word if $w$ visits each edge of $G_w$ at least twice. Assume that $w$ is a weak Wigner word. Since the graph $G_w = (V_w, E_w)$ of $w$ is connected and each edge is visited at least twice, it follows that 
\[
	\wt(w) = \# V_w  \le 1 + \# E_w \le 1 + \frac{\l(w) - 1}{2} = \frac{\l(w) + 1}{2}.
\]
A weak Wigner word $w$ of weight $\wt(w) = (\l(w)+1)/2$ is called a Wigner word. We also call a single letter word a Wigner word. Note that $w$ is a Wigner word only if its length is an odd number.

Here are some properties of a Wigner word $w$ (see \cite{Anderson2006} or \cite[Section~2.1]{Anderson2010} for more details):
\begin{itemize}
	\item[(i)] the graph $G_w$ is a tree, that is, a connected graph with no loop; 
	\item[(ii)] the set of self edges $E_w^s$ is empty;
	\item[(iii)] the path $w$ visits each connecting edge exactly twice, $N_e^w = 2$ for all $e \in E_w$. 
\end{itemize}

A pair of words $(w_1, w_2)$ is called a weak CLT pair if 
\begin{itemize}
	%\item $w_1, w_2 \in \U$;
	\item[(P1)] $N_e^a \ge 2$, for all $e \in E_a$, where $a = (w_1, w_2)$;
	\item[(P2)] $E_{w_1} \cap E_{w_2} \neq \emptyset$.
\end{itemize}

To study properties of weak CLT pairs, we need the following simple but useful property. It is a special case of the so called ``the parity principle'' (see \cite[Lemma~4.4]{Anderson2006}). 
\begin{lem}[Closed walk on a tree]
A closed walk on a tree visit each edge an even of times.
\end{lem}

\begin{lem}\label{lem:weak-CLT-pair}
Let $a=(w_1, w_2)$ be a weak CLT pair. Then 
\[
	\wt(a) \le  \frac{\l(w_1) + \l(w_2)}{2} -1.
\]
\end{lem}
\begin{proof}
Let $G_a = (V_a, E_a)$ be the graph of the sentence $a$. Since the pair $(w_1,w_2)$ visits each edge at least twice, it follows that 
\[
	\# E_a \le \frac{\l(w_1) -1 + \l(w_2) - 1}{2}.
\]
In addition, $\wt(a)\le 1+ \# E_a$ because the graph $G_a$ is connected.

Now, if $\wt(a) \le \# E_a$, then the conclusion immediately follows. Thus, we only need to consider the case $\wt(a) = 1+ \# E_a$, in which $G_a$ is a tree. Since $w_1, w_2$ are closed walks on the tree $G_a$, each word $w_1, w_2$ visits any edge $e \in E_a$ an even of times. Consequently, a common edge of $w_1$ and $w_2$ is visited at least four times. Therefore, 
\[
	\# E_a \le \frac{\l(w_1) -1 + \l(w_2) - 1}{2} - 1,
\]
and hence the conclusion follows.
\end{proof}

A pair $(w_1, w_2)$ is called a CLT pair if it is a weak CLT pair and in addition, 
\[
	\wt((w_1, w_2)) = \frac{\l(w_1) + \l(w_2)}{2} - 1.
\]

Denote by $\U_{k_1, k_2}$ a set of representatives for equivalence classes of CLT pairs $(u_1, u_2)$, where $u_1$ and $u_2$ are $(k_1 + k_2)/2 $-words of length $k_1$ and $k_2$, respectively, provided that $k_1 + k_2$ is even. When $k_1 + k_2$ is odd, we set  $\U_{k_1, k_2}= \emptyset$.

The following lemma introduces some properties of CLT pairs. We omit an easy proof.
\begin{lem}\label{lem:CLT-pair}
Let $a =(u_1, u_2) \in \U_{k_1, k_2}$ with $k_1, k_2 \ge 2$, and $k_1 + k_2$ being even. Then either $\wt(a) = 1 + \# E_a$ or $\wt(a) = \# E_a$. Moreover, the following hold.
\begin{itemize}
	\item[\rm(i)] If $\wt(a) = 1 + \# E_a$, then $G_a$ is a tree and 
		\begin{itemize}
			\item[\rm(a)] $N_e^{u_i} = 2$, for all $e \in E_{u_i}, i = 1,2$;
			\item[\rm(b)] $N_e^{a} = 2$, for all $e \in E_a$ except one edge $e_0$ with $N_{e_0}^a = 4$.
		\end{itemize}
	\item[\rm(ii)] If $\wt(a) = \# E_a$, then 
		\begin{itemize}
			\item[\rm(a)] $N_e^{u_i} = 1$, for some $e \in E_{a}, i = 1,2$;
			\item[\rm(b)] $N_e^{a} = 2$, for all $e\in E_a$.
		\end{itemize}
\end{itemize} 
\end{lem}

A sentence $a = (w_1,\dots,w_n)$ is called a weak CLT sentence if the following conditions hold
\begin{itemize}
	%\item[(S0)] $w_i \in \U, i =1,2,\dots,n$;
	\item[(S1)] $N_e^a \ge 2,$ for all $e \in E_a$;
	\item[(S2)] for all $i$, there exists $j \neq i$ such that $E_{w_i} \cap E_{w_j} \neq \emptyset$.
\end{itemize}

\begin{lem}\label{lemma:CLTword-1}
Let $a = (w_1,\dots,w_n)$ be a weak CLT sentence. Then 
\[
	\wt (a) \le 1 + \sum_{i = 1}^n \frac{\l(w_i) - 2}{2}.
\] 
\end{lem}

A sentence $a = (w_1,\dots,w_n)$ is called a CLT sentence if $a$ is a weak CLT sentence and the above equality holds, namely,
\[
	\wt (a) = 1 + \sum_{i = 1}^n \frac{\l(w_i) - 2}{2}.
\]

\begin{lem}\label{lemma:CLTword-2}
Let $a = (w_1,\dots,w_n)$ be a CLT sentence with $w_i \in \U, i = 1,2,\dots,n$. Then the following hold.
\begin{itemize}
	\item[\rm(i)] For each $i$, there exists unique $j \neq i$ such that $E_{w_i} \cap E_{w_j} \neq \emptyset$. 
	\item[\rm(ii)] The number $n$ is even and there exists a perfect matching $\sigma \in S_n$ such that
		\begin{itemize}
			\item[\rm (a)]	$a_i = (w_{\sigma(2i - 1)}, w_{\sigma(2i)})$ is a CLT pair, $i = 1,2,\dots, n/2$;
			\item[\rm (b)] $\{E_i\}_{i = 1}^{n/2}$ are disjoint sets, where $G_{i} = (V_i, E_i)$ denotes the graph of $a_i$;
			\item[\rm (c)] $\{\{V_i \setminus \{1\} \}\}_{i = 1}^{n/2}$ are disjoint sets.
		\end{itemize}
\end{itemize}
\end{lem}

\begin{proof}[Proof of Lemma~\ref{lemma:CLTword-1}]
This lemma is a special case of \cite[Lemma~2.1.34]{Anderson2010}. However, we mention the proof here because it will be used in the next lemma. Let $a=(w_1, w_2, \dots, w_n)$ be a weak CLT sentence, where $w_i = \{s_{i,j}\}_{j = 1,\dots, \l(w_i)}$. Let $I = \cup_{i = 1}^n \{i\} \times \{1,2,\dots,\l(w_i) - 1\}$ and let $A$ be an $n$ rows left-justified table whose entries are the edges of $a$, namely, 
\[
	A_{ij} = (s_{i,j}, s_{i, j+1}), \quad(i,j) \in I.
\]

Let $G_a = (V_a, E_a)$ be the graph of the sentence $a$. Note that $G_a$ is a connected graph because every word is a closed word starting at $1$.  Let $G' = (V', E')$ be any spanning tree in $G_a$. Then we have $\wt(a) = 1 + \# E'$ and so in order to proof the lemma, we just have to bound $\#E'$.

Now let $X = \{X_{ij}\}_{(i,j) \in I}$ be a table of the same ``shape'' as $A$, but with all entries equal either to $0$ or $1$. We call $X$ an edge-bounding table if the following conditions hold:
\begin{itemize}
	\item[(E1)] for all $(i,j) \in I$, if $X_{ij} = 1$, then $A_{ij} \in E'$;
	\item[(E2)] for each $e \in E'$, there exist distinct $(i_1, j_1), (i_2,j_2) \in I$ such that $X_{i_1, j_1} = X_{i_2, j_2} = 1$ and $A_{i_1, j_1} = A_{i_2, j_2} = e$;
	\item[(E3)] for each $e \in E'$ and index $i \in \{1,\dots,n\}$, if $e$ appears in the $i$th row of $A$, then there exists $(i,j) \in I$ such that $A_{ij} = e$ and $X_{ij} = 1$.
\end{itemize}
 
For an edge-bounding table $X$, the corresponding quantity $\frac 12 \sum_{(i,j) \in I} X_{ij}$ bounds $\# E'$, whence the terminology. At least one edge-bounding table exists, namely the table with a $1$ in position $(i,j)$ for each $(i,j) \in I$ such that $A_{ij} \in E'$ and $0$'s elsewhere. Now let $X$ be an edge-bounding table such that for some index $i_0$ all the entries of $X$ in the $i_0$th row are equal to $1$. Then all egdes of $w_{i_0}$ belongs to $E'$. In other words, $w_{i_0}$ is a closed walk in the tree $G'$, hence every entry in the $i_0$th row of $A$ appears there an even number of times and a {\it fortiori} at least twice. Now choose $(i_0, j_0) \in I$ such that $A_{(i_0,j_0)} \in E'$ appears in more than one row of $A$. Let $Y$ be the table obtained by replacing the entry $1$ of $X$ in position $(i_0,j_0)$ by the entry $0$. Then it is not difficult to check that $Y$ is again an edge-bounding table. Proceeding in this way we can find an edge-bounding table with $0$ appearing at least once in every row, and hence we have 
\[
	\#E' \le \frac 12 (\#I - n) = \frac {\sum_{i = 1}^n (\l(w_i) - 2)}{2}.
\]
The lemma is proved.
\end{proof}

\begin{proof}[Proof of Lemma~\ref{lemma:CLTword-2}](i) Assume that $a = (w_1, \dots, w_n)$ is a CLT sentence with $w_i \in \U, i = 1,2,\dots,n$. Let $G_a, G'$ be the graph of $a$ and the spanning tree as in the proof of Lemma~\ref{lemma:CLTword-1}. Moreover, let $X$ be an edge-bounding table satisfying the condition that at least one entry is $0$ in each row. Then, recall that 
\[
	\# E' \le \frac{1}{2} \sum_{(i,j) \in I} X_{ij} \le \frac {\sum_{i = 1}^n (\l(w_i) - 2)}{2}.
\]
Therefore, the above two inequalities must become equalities by the definition of CLT sentence. Consequently, the edge-bounding table $X$ has exactly one $0$-entry in each row. For each $i$, let $e_i$ denote the edge $A_{ij}$ at the position $X_{ij} = 0$. Note that by the first property (property (E1)) of the edge-bounding table $X$,
\begin{itemize}
\item [(*)] all edges of $w_i$, except at most one edge $e_i$, belong to $E'$.
\end{itemize} 
We claim that for each $i$, there is a unique $\check i \neq i$ such that $e_{\check i} = e_i$. This claim is shown as follows.

Let 
\[
	N'_e := \# \{(i,j) \in I : X_{ij} = 1, A_{ij} = e\}. 
\]
Then the two equalities imply that $N'_e = 2$ for all $e \in E'$.  

\emph{Uniqueness.} Assume that there are at least three words $w_{i_1}, w_{i_2}, w_{i_3}$ such that $e_{i_1} = e_{i_2} = e_{i_3} = (s,\bar s)$. Since we consider words which do not contain self edge, assume without loss of generality that $s \neq 1$. Then each word $w_{i_k}$ contains a walk on the tree $G'$ from $1$ to $s$ (or from $s$ to $1$), which can be chosen to traverse only those edges $A_{i_k,j}$ with $X_{i_k,j} = 1$. Therefore, there exists some edge $e$ with $N'_e \ge 3$, which is a contradiction.

\emph{Existence.} Now fix some index $i$. Then either $e_i \not \in E'$ or $e_i \in E'$. 

Case 1: $e_i \not \in E'$. In this case, $N_e^{w_i} = 1$ by (*). Thus, $e_i \in E_{w_{i_1}}$ for some $i_1 \not = i$ because $N_e^a \ge 2$ (see property (S1)). It also follows from (*) that $e_{i_1} = e_i$. Assume that $e_i = (s, \bar s)$ and $w_i$ is a walk $1 \to s \to \bar s \to 1$. The word $w_{i_1}$ may be either $1 \to s \to \bar s \to 1$ or $1 \to \bar s \to s \to 1$. We construct a new word/walk $w_i \vee w_{i_1}$ as follows. Walk from $1$ to $s$ by $w_i$, then go to $\bar s$ by $w_{i_1}$, an back to $1$ by $w_i$. A new word $w_i \vee w_{i_1}$ of length $\l(w_i) + \l(w_{i_1}) - 3$ is a closed walk on a tree $G'$, and thus $N_e^{w_i \vee w_{i_1}}$ is even, and hence is at least $2$. It follows that $N_e^{w_i \vee w_{i_1}} = 2$ because it is bounded by $N'_e$.  
 
Case 2. $e_i \in E'$. In this case, $w_i$ is a closed walk on the tree $G'$, which implies that $N_{e_i}^{w_i}$ is even. Moreover, it is bounded by $1 + N'_e = 3$. Thus $N_{e_i}^{w_i} = 2$. Therefore, in the $i$th row, there is only one pair $(i,j)$ such that $X_{i,j} = 1$ and $A_{ij} = e_i$. By property (E2) of edge-bounding table, there is another pair $(i_1, j_1)$ such that $X_{i_1,j_1} = 1$ and $A_{i_1, j_1} = e_i$. Note that $i_1 \neq i$. 

Next, we show that $e_{i_1} = e_{i}$. Indeed, assume to the contrary that $e_{i_1} \neq e_i$. There are two cases to consider.
\begin{itemize}
\item if $e_{i_1} \in E'$, then by the same argument as in the beginning of case~2, it follows that $N_{e_i}^{w_{i_1}} = 2$, therefore $N'_{e_i} \ge 3$, which is a contradiction;
\item if $e_{i_1} \not\in E'$, then by case 1, there exists $i_2$ with $e_{i_2} = e_{i_1}$ and $N_{e_i}^{w_{i_1}\vee w_{i_2}} = 2$. It also follows that $N'_e \ge 3$, the same contradiction.
\end{itemize}
We also construct a new word/walk $w_i \vee w_{i_1}$ as in case 1. 

\noindent (ii) It is clear that $n$ must be an even number because $n$ words $w_1,\dots, w_n$ can be partition in pairs which have the same $e_i$. We construct a permutation $\sigma$ on $\{1,2,\dots,n\}$ as follows. Let 
\[
	\begin{cases}
		\sigma(1) = 1, &\\
		\sigma(2) = j, &\text{if $(w_1, w_j)$ is a pair.} 
	\end{cases}
\]
Then by induction, we define for $i = 2,3,\dots, n/2$,
\[
	\begin{cases}
		\sigma(2i+1) = \min \{\{1, \dots, n\} \setminus \{\sigma(1),\dots, \sigma(2i)\}, &\\
		\sigma(2i + 2) = j, &\text{if $(w_{\sigma(2i+1)}, w_j)$ is a pair.} 
	\end{cases}
\]
It is clear that $\sigma$ is a perfect matching. Moreover words/walks $w_{\sigma(2i - 1)} \vee w_{\sigma(2i)}$ are distinct walks on the tree $G'$. The rest of lemma follows. 
\end{proof}

\section{The Wigner semicircle law for spectral measures}
In this section, we will show that spectral measures of Wigner matrices also converge weakly, in probability, to the semicircle distribution.
Recall that $\{\xi_{ij}\}_{1\le i\le j}$ are independent real random variables with the following properties:
\begin{itemize}
	\item[(i)] $\{\xi_{ii}\}_{1\le i}$ is an i.i.d.\ sequence with $\E[\xi_{11}] = 0$ and $\E[|\xi_{11}|^p]<\infty, p = 2,3,\dots$;
	\item[(ii)] $\{\xi_{ij}\}_{1\le i < j}$ is another i.i.d. sequence with $\E[\xi_{12}]=0, \E[\xi_{12}^2] = 1$ and $\E[|\xi_{12}|^p] < \infty, p = 3,4,\dots.$
\end{itemize} 
Recall also that the Wigner matrix $X_N$ is defined as 
\[
	X_N(i,j) = X_N(j,i) = \frac{\xi_{ij}}{\sqrt{N}}, \quad 1 \le i \le j \le N.
\]

We begin with the following expression for $X_N^k (1,1)$,
\begin{align*}
	X_N^k (1,1) &= \sum_{i_1, i_2, \dots, i_{k - 1} = 1}^N  X_{1,i_1} X_{i_1, i_2} \cdots X_{i_{k - 1},1} \\
	&= \frac{1}{N^{\frac k2}} \sum_{i_1, i_2, \dots, i_{k - 1} = 1}^N  \xi_{(1,i_1)} \xi_{(i_1, i_2)} \cdots \xi_{(i_{k - 1},1)} \\
	&= \frac{1}{N^{\frac k2}} \sum_{w \in \W_k^{(N)}} T_w,
\end{align*}
where $T_w = \prod_{e \in E_w} \xi_e^{N_e^w}$.  

\begin{lem}\label{lemma:convergence-of-moments}
\begin{itemize}
	\item[\rm (i)]For odd  $k$, 
		\[
			\E[X_{N}^k(1,1)] \to 0 \text{ as }N \to \infty.
		\]
	\item[\rm (ii)] For even $k$,
		\[
			\E[X_{N}^k(1,1)] \to C_{k/2} \text{ as }N \to \infty,
		\]
	where $C_{n}$ denotes the $n$th Catalan number, 
	\[
		C_{n} = \frac{\begin{pmatrix} 2n \\ n\end{pmatrix}}{n + 1} = \frac{(2n)!}{(n+1)! n !},
	\]
	which is the numbers of equivalence classes of Wigner words of length $2n + 1$.
\end{itemize}
\end{lem} 
\begin{proof}
It is clear that
	\[
		\E[X_{N}^k(1,1)] = \frac{1}{N^{k/2}}\sum_{w \in \W_k^{(N)}} \E[T_w]. 
	\]
Recall that $T_w = \prod_{e \in E_w} \xi_e^{N_e^w}$, which implies that $\E[T_w] = \prod_{e \in E_w} \E[\xi_e^{N_e^w}]$. Thus $E[T_w] = 0$ unless $w$ is a weak Wigner word.
	
	Let $\W_{k;t}$ denotes a set of representatives for equivalence classes of weak Wigner words $w \in \W_k^{(t)}$ of weight $t$. Then for $N \ge t$, given a word $w \in \W_{k,t}$, there are exactly 
\[
	C_{N,t} := (N-1)(N-2) \cdots (N - t + 1)
\]
words in $\W_k^{(N)}$ that are equivalent to $w$.

Since the weight of a weak Wigner word of length $k + 1$ is bounded by $(k/2 + 1)$, and two equivalent words have the same graphs, we can rewrite the expression of $\E[X_{N}^k(1,1)]$ as
\begin{align*}
	\E[X_{N}^k(1,1)] &= \frac{1}{N^{k/2}} \sum_{t \le \frac k2 + 1} \sum_{w \in \W_{k;t}} \sum_{w' \in \W_k^{(N)} : w' \sim w} \E[T_{w'}]\\
	&= \frac{1}{N^{k/2}}\sum_{t \le \frac k2 + 1} C_{N,t} \sum_{w \in \W_{k;t}} \E[T_w]\\
	&= \sum_{t \le \frac k2 + 1} \frac{C_{N,t}}{N^{k/2}} \sum_{w \in \W_{k;t}} \E[T_w].
\end{align*}
Note that as $N \to \infty$, $C_{N,t}/ N^{t-1} \to 1$. Note also that the cardinality of $\W_{k;t}$ is finite and that $\E[T_w] < \infty$ because  all moments of $\{\xi_{ij}\}$ are finite. Therefore, as $N \to \infty$, 
\[
	\E[X_{N}^k(1,1)] \to \begin{cases}
		0, &\text{if $k$ is odd,}\\
		\sum_{w \in \W_{k; k/2 + 1}} \E[T_w], &\text{if $k$ is even}.
	\end{cases}
\] 
Finally, $w \in \W_{k; k/2 +1}$ means that $w$ is a Wigner word, and hence $\E[T_w] = 1$ by properties of Wigner words. Thus for even number $k$, the limit of $\E[X_{N}^k(1,1)]$ is equal to the number of equivalence classes of Wigner words of length $k+1$, which is nothing but the $(k/2)$th Catalan number. The lemma is proved.
\end{proof}

\begin{lem}\label{lemma:convergence-in-L2-of-moments}It holds that
\[
	\E[(X_{N}^k(1,1) - \E[X_{N}^k(1,1)])^2] \to 0 \text{ as }N \to \infty.
\]
\end{lem}
\begin{proof}
We begin with the following expression
\[
	X_{N}^k(1,1) - \E[X_{N}^k(1,1)] = \frac{1}{N^{k/2}} \sum_{w \in \W_k^{(N)}} (T_w - \E[T_w]) =: \frac{1}{N^{k/2}} \sum_{w \in \W_k^{(N)}} \bar T_w.
\]
Here $\bar T_w := T_w - \E[T_w]$. Then 
\[
	(X_{N}^k(1,1) - \E[X_{N}^k(1,1)])^2 = \frac{1}{N^k} \sum_{w_1, w_2 \in \W_k^{(N)}}  \bar T_{w_1} \bar T_{w_2}  = \frac{1}{N^k} \sum_{w_1, w_2 \in \W_k^{(N)}}  \bar T_{(w_1, w_2)}, 
\]
where $\bar T_{(w_1, w_2)} :=  \bar T_{w_1} \bar T_{w_2}$. 

It is clear that $\E[\bar T_{(w_1, w_2)}] = 0$ unless $(w_1, w_2)$ is a weak CLT pair. Similar argument as in the proof of Lemma~\ref{lemma:convergence-of-moments} with noting that $\wt((w_1, w_2)) \le k $ if $(w_1, w_2)$ is a weak CLT pair, we have 
\[
	\E[(X_{N}^k(1,1) - \E[X_{N}^k(1,1)])^2]  = \sum_{t \le k} \frac{C_{N,t}}{N^k} \sum_{(w_1, w_2)\in \W_{k,k; t}} \E[\bar T_{(w_1, w_2)}].   
\]
Here $\W_{k,k; t}$ denotes a set of representatives for equivalence classes of weak CLT pair/sentence $(w_1, w_2)$ of weight $t$, where $w_1$ and $w_2$ are both $t$-words of length $k +1$. Therefore 
\[
	\E[(X_{N}^k(1,1) - \E[X_{N}^k(1,1)])^2] \to 0 \text{ as } N \to \infty,
\]
which completes the proof.
\end{proof}

As a direct consequence of Lemma~\ref{lemma:convergence-of-moments} and Lemma~\ref{lemma:convergence-in-L2-of-moments}, we have the following result.
\begin{lem}\label{lemma:convergence-in-probability-of-moments}
As $N \to \infty$, $X_{N}^k(1,1)$ converges in $L^2$, and hence, converges in probability to $\langle \sigma, x^k\rangle$.
\end{lem}

We are now in a position to investigate the semicircle law for spectral measures of Wigner matrices.
\begin{defn} Let $A$ be a real symmetric matrix of degree $N$ and $v$ be a unit vector in $\R^N$. Then the spectral measure $\mu$ of $(A, v)$ is the probability measure on $\R$ satisfying 
\[
	\int_\R x^k \mu(dx) = (A^k v, v), \quad k=0,1,2,\dots,
\]
where $\left(\cdot,\cdot \right)$ denotes the inner product in $\R^N$.
\end{defn}

Let $A$ be a real symmetric matrix. Let $\lambda_1 \le \lambda_2 \le \cdots \le \lambda_N$ be the eigenvalues of $A$, and let $v_1, v_2, \dots, v_N$ be corresponding eigenvectors which are chosen to be an orthonormal system of $\R^N$. Then the spectral decomposition of $A$ can be written as
\[
	A = \sum_{j = 1}^N \lambda_j v_j  v_j^T.
\] 
Consequently, 
\[
	A^k = \sum_{j = 1}^N \lambda_j^k v_j  v_j^T,
\]
and thus, 
\[
	(A^k v, v) = \sum_{j = 1}^N \lambda_j^k (v, v_j)^2.
\]
Therefore, the spectral measure of $(A, v)$ is given by
\[
	\mu = \sum_{j = 1}^N (v, v_j)^2 \delta_{\lambda_j}.
\]

Now let $\nu_N$ be the spectral measure of $(X_N, e_1)$, where $e_1 = (1,0,\dots,0)^T \in \R^N$. Then by definition,  
\[	
	\langle  \nu_N, x^k \rangle = (X_N^k e_1, e_1) = X_N^k(1,1).
\]

\begin{thm}
	\begin{itemize}	
		\item[\rm(i)] The $k$th moment of $\nu_N$ converges in probability to that of the semicircle law, namely, 
		\[
			\langle \nu_N, x^k\rangle \to  \langle \sigma, x^k\rangle \text{ in probability as $N \to \infty$}.
		\]
		\item[\rm(ii)] The spectral measure $\nu_N$ converges weakly, in probability, to the semicircle distribution.
	\end{itemize}
\end{thm}
\begin{proof}
The statement (i) is just Lemma~\ref{lemma:convergence-in-probability-of-moments}. 
	
Since $\sigma$ has compact support, we will show that (ii) follows from (i). Indeed, 
let $f$ be a bounded continuous function on $\R$. We need to prove that 
\[
	\langle \nu_N, f \rangle \to \langle \sigma, f \rangle \text{ in probability as $N \to \infty$.} 
\] 

Recall that $\sigma$ is supported in $[-2,2]$, which implies that $\langle \sigma, x^{2k}\rangle \le 2^{2k}$. Let $B > 2$ be fixed. Then, for $k = 0,1,\dots$,  
\begin{align*}
	|\langle \nu_N , x^k {\bf 1}_{\{|x|> B\}}\rangle| &= \left|\int_{\R} x^k {\bf 1}_{\{|x|> B\}} d\nu_N(x) \right| \\
	&\le \int_{\R} |x|^k {\bf 1}_{\{|x|> B\}} d\nu_N(x) \\
	&\le \frac{1}{B^{2n - k}} \int_{\R} x^{2n}  d\nu_N(x) = \frac{\langle \nu_N, x^{2n} \rangle}{B^{2n - k}}, \text{ for $k < 2n$}.
\end{align*}
By letting $N \to \infty$, we obtain
\[
	|\langle \nu_N , x^k {\bf 1}_{\{|x|> B\}}\rangle| \le \frac{\langle \nu_N, x^{2n} \rangle}{B^{2n - k}} \underset{\text{as $N\to \infty$ }}{\overset{\text{in probability}}{\longrightarrow }}   \frac{\langle \sigma, x^{2n} \rangle}{B^{2n - k}}  \le \frac{2^{2n}}{B^{2n - k}}.
\]
Note that $2^{2n}/B^{2n - k} \to 0$ as $n \to \infty$. Thus 
\[
	\langle \nu_N , x^k {\bf 1}_{\{|x|> B\}} \rangle  \to 0 \text{ in probability as $N \to \infty$.}
\]
Consequently, for any polynomial $Q$, 
\begin{equation}\label{convergence-in-probability-of-polynomial}
	\langle \nu_N , Q {\bf 1}_{\{|x|> B\}} \rangle \to 0 \text{ in probability as $N \to \infty$.}
\end{equation}

Given $\varepsilon > 0$, there is a polynomial $Q$ such that 
\[
	\sup_{|x| \le B} |f(x) - Q(x)| \le \varepsilon.
\]
Then consider the following decomposition   
\begin{align*}
	\langle \nu_N, f\rangle - \langle \sigma, f \rangle = \langle \nu_N, f {\bf 1}_{\{|x|> B\}}\rangle + \langle \nu_N, (f - Q){\bf 1}_{\{|x|\le B\}}\rangle \\
	- \langle \nu_N, Q {\bf 1}_{\{|x|> B\}}\rangle + (\langle \nu_N, Q \rangle - \langle \sigma, Q\rangle) + \langle \sigma, Q- f\rangle. 
\end{align*}
The first term and the third term converges to $0$ in probability  by \eqref{convergence-in-probability-of-polynomial}. The fourth term converges to $0$ in probability by (i) of this theorem. Finally, the second term and the fifth term is bounded by $\varepsilon$. Since $\varepsilon$ is arbitrary, it follows that $\langle \nu_N, f\rangle $ converges to $\langle \sigma, f\rangle$ in probability. The proof is complete.
\end{proof}

\section{Central limit theorem for moments of spectral measures}
This section investigates weak limits of moments of spectral measures, more precisely, the weak limits of $\sqrt{N}(X_{N}^k (1,1) - \E[X_{N}^k (1,1)])$ as $N$ tends to infinity. 
\subsection{Zero diagonal}
Recall that 
\[
	X_{N}^k(1,1) = \frac{1}{N^{\frac k2}}\sum_{w \in \W^{(N)}_k} T_w, 
\]
where $T_w = \prod_{e \in E_w} \xi_e^{N_e^w}$. 

Let 
\begin{align*}
	Y_{N,k} &:= \sqrt{N}\left( \frac{1}{N^{\frac k2}}\sum_{w \in \U^{(N)}_k} (T_w - \E[T_w]) \right) \\
	&= \frac{1}{N^{\frac{k-1}{2}}}\sum_{w \in \U^{(N)}_k} \bar T_w \\
	&\left(= \sqrt{N} (X_{N}^k(1,1) - \E[X_{N}^k(1,1)]), \text{ if }\xi_{11} = 0\right).
\end{align*}
For a sentence $a = (w_1,\dots,w_n)$, we denote 
\[
	\bar T_{a} = \bar T_{(w_1,\dots,w_n)} = \bar T_{w_1} \cdots \bar T_{w_n}.
\]

Next, we consider $\E[Y_{N,k_1} Y_{N,k_2}]$ for fixed $k_1, k_2 \ge 2$. It is clear that 
\[
	\E[Y_{N,k_1} Y_{N,k_2}] = \frac{1}{N^{\frac{k_1 + k_2}{2} - 1}} \sum_{w_1 \in \U^{(N)}_{k_1}, w_2 \in \U^{(N)}_{k_2}} \E[\bar T_{(w_1,w_2)}].
\]

\begin{lem}\label{lem:covariance} For $k_1, k_2 \ge 2$,
	\[
		\lim_{N \to \infty} \E[Y_{N,k_1} Y_{N,k_2}] = \sum_{(w_1, w_2) \in \U_{k_1, k_2}} \E[\bar T_{(w_1, w_2)}].
	\]
The limit is positive, if $k_1 + k_2$ is even, and only depends on the second and the fourth moments of $\xi_{12}$. It is zero, if $k_1 + k_2$ is an odd number. 
\end{lem}
\begin{proof}
	It is clear that $\E[\bar T_{(w_1,w_2)}] = 0$ unless $(w_1, w_2)$ is a weak CLT pair. Let $\U_{k_1, k_2}^{(t)}$ denote a set of representatives for equivalence classes of weak CLT pairs $(w_1, w_2)$ of weight $t$, where $w_1$ and $w_2$ are $t$-words of lengths $k_1 +1$ and $k_2 + 1$, respectively. By Lemma~\ref{lem:weak-CLT-pair}, $t \le (k_1 + k_2)/2$ unless $\U_{k_1, k_2}^{(t)} = \emptyset$. For $t = (k_1 + k_2)/2$, the set $\U_{k_1, k_2}^{(t)}$ is just a set of representatives for equivalence classes of CLT pairs $\U_{k_1, k_2}$. An argument similar to Lemma~\ref{lemma:convergence-of-moments}, we obtain
\[
	\lim_{N \to \infty} \E[Y_{N,k_1} Y_{N,k_2}] = 
	\begin{cases}
		0, &\text{if $k_1 + k_2$ is odd,}\\
		\sum\limits_{(w_1, w_2) \in \U_{k_1, k_2}} \E[\bar T_{(w_1, w_2)}], &\text{if $k_1 + k_2$ is even}.
	\end{cases}
\]

Next, let $(w_1,w_2) \in \U_{k_1, k_2}$. If $\wt(a) = 1+ \# E_a$, then by Lemma~\ref{lem:CLT-pair} (i), $\E[T_{w_i}] = 1, i = 1,2$. Moreover, $\E[T_{w_1} T_{w_2}] = \E[\prod_{e \in E_a} \xi_e^{N_e^a}] = \E[\xi_{e_0}^4] = \E[\xi_{12}^4]$, where $e_0$ is the only edge with $N_{e_0}^a = 4$. Thus 
\[
	\E[\bar T_{(w_1, w_2)}] = \E[T_{w_1}T_{w_2}] - \E[T_{w_1}] \E[T_{w_2}] = \E[\xi_{12}^4] - 1 \ge 0.
\] 
The last inequality holds because $\E[\xi_{12}^4] - 1 = \E[(\xi_{12}^2 - 1)^2] \ge 0$.

Now, if $\wt(a) = \# E_a$, then $\E[T_{w_i}] = 0$ because there exists an edge which is visited only one time by $w_i, i=1,2$. Further, since each edge is visited exactly two times by $(w_1, w_2)$, it follows that $\E[T_{w_1}T_{w_2}] = 1$. Combining those we have 
\[
	\E[\bar T_{(w_1, w_2)}] = 
	\begin{cases}
		1, &\text{if $\wt(a) = \# E_a $}, \\
		\E[\xi_{12}^4] - 1 \ge 0, &\text{if $\# \wt(a) = 1 + E_a$}.
	\end{cases}
\]
Finally, the set of CLT pairs $a = (w_1, w_2)$ for which $\wt(a) = \# E_a$ is not empty. Thus, the rest of this lemma follows.
\end{proof}

By an argument similar to the previous lemma, Lemma~\ref{lemma:CLTword-1} implies the following statement.
\begin{lem}\label{lem:CLT-sentence}
For $k_1, k_2, \dots, k_n \ge 2$, 
\[
	\lim_{N \to \infty} \E\left[\prod_{i = 1}^n Y_{N,k_i} \right] = \sum_{(w_1,\dots,w_n) \in \U_{k_1,\dots,k_n}} \E[\bar T_{(w_1, \dots, w_n)}].
\]
Here $\U_{k_1,\dots,k_n}$ denotes a set of representatives for equivalence classes of CLT sentences $a = (w_1,\dots,w_n)$, where $w_i \in \U_{k_i}^{(t)}$, $t = 1 + \sum_{i = 1}^n \frac{k_i - 1}{2}$. 
\end{lem}

Let 
\[
	A(k_1, k_2) := \sum_{(w_1, w_2) \in \U_{k_1, k_2}} \E[\bar T_{(w_1, w_2)}] .
\]
Then the matrix $(A(k,l))_{k,l = 2,3,\dots}$ is symmetric. Each finite block $(A(k,l))_{k,l = 2}^n$ is positive semidefinite because it is the limit of the covariance matrix of random variables $(Y_{N,k})_{k = 2,\dots,n}$. Thus, there exists a sequence of mean zero jointly Gaussian random variables $\{\eta_k\}_{k = 2,3,\dots}$ defined on the same probability space such that 
\[
	\E[\eta_k \eta_l] = A(k,l).
\]

\begin{lem}
For even number $n$, 
\begin{equation}\label{limit-of-product}
	\sum_{(w_1,\dots,w_n) \in \U_{k_1,\dots,k_n}} \E[\bar T_{(w_1, \dots, w_n)}] = \sum_{\substack{\sigma \in S_n \\ \sigma \text{: perfect matching }}} \prod_{i = 1}^{n/2} A(k_{\sigma(2i -1)}, k_{\sigma(2i)}).
\end{equation}

\end{lem}
\begin{proof}
	It is a direct consequence of Lemma~\ref{lemma:CLTword-2}.
\end{proof}

\begin{thm}\label{thm:joint-distribution-of-Y}
The joint distribution of $\{Y_{N,k}\}_{k = 2}^K$ converges to that of $\{\eta_k\}_{k = 2}^K$ as $N$ tends to infinity for any fixed $K \ge 2$.
\end{thm}
\begin{proof}
	The left hand side of \eqref{limit-of-product} is exactly the Wick formula for the expectation 
\[
	\E\left[\prod_{i = 1}^n \eta_{k_i}\right].
\] 
Thus, for any even number $n$, and for any $k_1, \dots, k_n \ge 2$, 
\[
	\lim_{N \to \infty} \E\left[\prod_{i = 1}^n Y_{N,k_i} \right] = \E\left[\prod_{i = 1}^n \eta_{k_i}\right].
\]
This also holds if $n$ is odd, in which both sides are zero. Therefore, the joint distribution of  $\{Y_{N,k}\}_{k = 2}^K$ converges to that of $\{\eta_k\}_{k = 2}^K$  because Gaussian distributions are characterized by their moments.  
\end{proof}

\subsection{General case}

Let 
\[
	Z_{N,k} = \frac{1}{N^{\frac{k - 1}{2}}} \sum_{w \in \V_{k}^{(N)}} (T_{w} - \E[T_{w}]) = \frac{1}{N^{\frac{k - 1}{2}}} \sum_{w \in \V_{k}^{(N)}} \bar T_w.
\]
It is clear that $\E[Z_{N, k}] = 0$. We consider 
\[
	\E[Z_{N,k}^2] = \frac{1}{N^{k-1}} \sum_{w_1, w_2 \in \V_{k}^{(N)}}\E[ \bar T_{(w_1, w_2)}].
\]

Recall that $(w_1, w_2)$ is a weak CLT pair if 
\begin{itemize}
	\item[(P1)] $N_e^a \ge 2$, for all $e \in E_a$, where $a = (w_1, w_2)$;
	\item[(P2)] $E_{w_1} \cap E_{w_2} \neq \emptyset$.
\end{itemize}
For a word $w \in \V$, let $\check w \in \U$ be the word constructed from $w$ by deleting every adjacent same letter. Then the graph of $\check w$ is obtained from that of $w$ by removing all self edges. The following lemma refines Lemma~\ref{lem:weak-CLT-pair}

\begin{lem}\label{lem:CLT-U-pair}
	Let $w_1, w_2 \in \V_k$ be a weak CLT pair. Then
\begin{itemize}
	\item[\rm (i)]	$\wt((w_1, w_2)) \le k$, if $k$ is odd;
	\item[\rm (ii)] $\wt((w_1, w_2)) \le k - 1$, if $k$ is even.
\end{itemize}
\end{lem}
\begin{proof}
The proof is similar to that of Lemma~\ref{lem:weak-CLT-pair}. Let $\check w_1, \check w_2 \in \U$ be the words obtained from $w_1, w_2$ by deleting every adjacent same letter. Let $\check a = (\check w_1, \check w_2)$. Then $N_e^{\check a} \ge 2$ for all $e \in E_{\check a}$. Let $G_{\check a} = (V_{\check a}, E_{\check a})$ be the graph of $\check a$. Note that $G_{\check a}$ is connected because both $\check w_1$ and $\check w_2$ are words started from $1$. Note also that $\wt (a) = \wt(\check a)$. Since $N_e^{\check a} \ge 2$ for all $e \in \E_{\check a}$, it follows that 
\[
	\# E_{\check a} \le \frac{1}{2}(l(\check w_1) - 1 + l(\check w_2) - 1) \le \frac12 (l(w_1) -2 + l(w_2) -2) = k - 1.
\]
The last inequality holds because $l(\check w_i) \le l(w_i) - 1 = k, i = 1,2$. Thus 
\[
	\wt(a) = \wt(\check a) \le 1 + \# E_{\check a}  \le k.
\]

Next, we show that $\wt(a) = k$ does not hold if $k$ is even. Indeed, assume that $\wt(a) = k$. It follows that $\wt(a) = \wt(\check a) = 1 + \# E_{\check a}$, and hence the graph $G_{\check a}$ is a tree. In this case, it also implies that $l(\check w_i) = \l(w_i) - 1 = k, i = 1,2$. Thus $\check w_i$ is a closed walk of length $k$, which is even, on the tree $G_{\check a}$, which is impossible. The lemma is proved. 
\end{proof}

Let $\V_{k,k;t}$ denote a set of representatives for equivalence classes of weak CLT pairs $(w_1, w_2)$, where $w_1, w_2 \in \V$ are $t$-words of length $k + 1$. Then similarly to Lemma~\ref{lem:weak-CLT-pair}, we can show that 
\begin{equation}\label{variance-of-Z}
	\lim_{N \to \infty} \E[Z_{N,k}^2] = \sum_{(w_1, w_2) \in \V_{k,k;k}} \E[\bar T_{(w_1, w_2)}],
\end{equation}
which is zero if $k$ is even.

For odd $k$, let $\A_k$ denote a set of representatives for equivalent classes of words $w$ of length $k + 1$, for which $N_{(1,1)}^w = 1$ and $\check w$ is a Wigner word. Let $a_k$ be the cardinality of $\A_k$.

\begin{lem}
	Let $k \ge 3$ be an odd number. Let $(w_1, w_2) \in \V_{k,k;k}$. Then the following hold.
	\begin{itemize}
		\item[\rm (i)] $w_i$ is equivalent to some element of $\A_k, i = 1,2$.
		\item[\rm (ii)] $\supp(w_1) \cap \supp(w_2) = \{1\}$.
		\item[\rm (iii)] $\E[\bar T_{w_1, w_2}] = \E[\xi_{11}^2]$.
		\item[\rm (iv)] 
			\begin{equation}\label{variance-of-Z-2}
				\sum_{(w_1, w_2) \in \V_{k,k;k}} \E[\bar T_{(w_1, w_2)}] = a_k^2 \E[\xi_{11}^2].
			\end{equation}
	\end{itemize}
\end{lem}
\begin{proof}
	Let $\check a$ be as in the proof of Lemma~\ref{lem:CLT-U-pair}. Recall that, in this case, both $\check w_1$ and $\check w_2$ are walks of length $k$ on the tree $G_{\check a}$ and $\check w_i$ visit each of it edges exactly twice, $i = 1,2$. Thus $\check w_1$ and $\check w_2$ are Wigner words. Moreover, $N_{e}^{\check a} = 2$ for all $e \in \check a$, which implies that $\supp(\check w_1) \cap \supp(\check w_2) = \{1\}$. Now, it follows from the condition (P2), $E_{w_1} \cap E_{w_2} \neq \emptyset$, that $(1,1)$ must be a common edge of $w_1$ and $w_2$. Therefore, we obtain (i) and also (ii).

(iii) and (iv) are direct consequences of (i) and (ii). 
\end{proof}

\begin{lem}\label{lem:limit-of-Zk}
	Let $k$ be an odd number. Then the following hold. 
	\begin{itemize}
		\item[\rm(i)]
			\[
				\lim_{N \to \infty} \E[\xi_{11} Z_{N,k}] = a_k \E[\xi_{11}^2].
			\]
		\item[\rm(ii)]
			\[
				\lim_{N \to \infty} \E[(Z_{N,k} - a_k \xi_{11})^2] = 0.
			\]
	\end{itemize}
\end{lem}
\begin{proof}(i) 
	It follows from the definition of $Z_{N,k}$ that
	\[
		\E[\xi_{11} Z_{N,k}] = \frac{1}{N^{\frac{k - 1}{2}}} \sum_{w \in \V_k^{(N)}} \E[\xi_{11} \bar T_{w}].
	\]
	It is clear that $\E[\xi_{11} \bar T_{w}] = 0$ unless a word $w$ satisfies the following conditions
	\begin{itemize}
		\item $N_{(1,1)}^w \ge 1$;
		\item $N_{e}^w \ge 2$ for all $e \in E_{w}\setminus \{(1,1)\}$.
	\end{itemize}

Assume that a word $w$ satisfies the above conditions. Let $\check w$ be the simplified word of $w$. Then $\check w$ is a word of length at most $k$, which visits each edge at least twice. Thus, 
\[
	\wt(w) = \wt(\check w) \le \# E_{\check a} + 1  \le \frac{k - 1}{2}+ 1 = \frac{k + 1}{2}.
\] 
The equality $\wt(w) = (k + 1)/2$ holds if $\check w$ is a Wigner word of length $k$, or equivalently, if $w$ is equivalent to some word in $\A_k$. 

Now by a standard argument as in the proof of Lemma~\ref{lemma:convergence-of-moments}
\[
	\lim_{N \to \infty} \E[\xi_{11} Z_{N,k}] = \sum_{w \in \A_k} \E[\xi_{11} \bar T_{w}] = a_k \E[\xi_{11}^2].
\]

(ii) follows from (i), the limit \eqref{variance-of-Z} and the expression \eqref{variance-of-Z-2}. The lemma is proved.
\end{proof}

The following results are direct consequences of the limit \eqref{variance-of-Z} with even $k$ and Lemma~\ref{lem:limit-of-Zk}(iii).
\begin{lem}\label{lem:limit-theorem-for-Z}
\begin{itemize}
\item[\rm (i)] For even $k$, $Z_{N,k}$ converges in probability to zero.
\item[\rm (ii)] For odd $k$, $Z_{N,k}$ converges in probability to $a_k \xi_{11}$.
\end{itemize}
\end{lem}

\begin{thm}
	Let $\zeta$ be a random variable which has the same distribution as $\xi_{11}$ and is independent of $\{\eta_k\}_{k \ge 2}$. Let $\bar S_{N,k} = \sqrt{N} (X_{N}^k(1,1) - \E[X_{N}^k(1,1)])$. Then the following holds.
\begin{itemize}
	\item[\rm (i)] For even $k$, 
		\[
			\bar S_{N,k}  \dto \eta_k \text{ as } N \to \infty.
		\]
	\item[\rm (ii)] For odd $k \ge 3$,
		\[
			\bar S_{N,k}  \dto a_k \zeta + \eta_k \text{ as } N \to \infty.
		\]
	\item[\rm (iii)] For fixed $K$, the joint distribution of $(\bar S_{N, 1}, \bar S_{N, 2}, \dots,\bar S_{N,K})$ converges to that of $(\zeta, \eta_2, a_3 \zeta + \eta_3, \dots)$.
	
\end{itemize}
\end{thm}
\begin{proof}
We only need to prove (iii). Let $a_1 = 1, Y_{N,1} = 0$ and $Z_{N,1} = \xi_{11}$.
	For even $k$, let $a_k = 0$. Note that  
\[
		\bar S_{N,k} = Y_{N,k} + Z_{N,k} = Y_{N,k} + a_k \xi_{11} + (Z_{N,k} - a_k \xi_{11}).
\]

For any real numbers $\{\alpha_k\}_{k = 1}^K$, we consider 
\begin{align*}
	\sum_{k = 1}^K {\alpha_k \bar S_{N,k}} &= \sum_{k = 2}^K \alpha_k Y_{N,k} + (\sum_{k = 1}^K \alpha_k a_k) \xi_{11} + \sum_{k= 2}^K \alpha_k (Z_{N,k} - a_k \xi_{11}) \\
	&=: S_1 + S_2 + S_3.
\end{align*}

As $N \to \infty$, $S_1 $ converges in distribution to $\sum_{k = 2}^K \alpha_k \eta_k$ by Theorem~\ref{thm:joint-distribution-of-Y}. Since $S_1$ is independent of $\xi_{11}$, it follows that $S_1 + S_2$ converges in distribution to $\sum_{k = 2}^K \alpha_k \eta_k + (\sum_{k = 1}^K \alpha_k a_k) \zeta$ as $N$ tends to infinity. Finally, $S_3$ converges in probability to zero by Lemma~\ref{lem:limit-theorem-for-Z}. Therefore, 
\[
	\sum_{k = 1}^K {\alpha_k \bar S_{N,k}} \dto \sum_{k = 1}^K \alpha_k (a_k \zeta + \eta_k) \text{ as } N \to \infty.
\]
The theorem is proved.
\end{proof}

{}

\hfill
\begin{tabular}{l}
Trinh Khanh Duy \\
Institute of Mathematics for Industry \\
Kyushu University\\
Fukuoka 819-0395, Japan \\
e-mail: trinh@imi.kyushu-u.ac.jp; duytkvn@gmail.com \\
\end{tabular}

\end{document}